\documentclass{article}
\usepackage{graphicx} 
\usepackage{amsmath}
\usepackage[shortlabels]{enumitem}
\usepackage[export]{adjustbox}
\usepackage{booktabs}
\usepackage[maxbibnames=99, style=alphabetic, maxalphanames=99, minalphanames=6, 
sorting=nyt, giveninits=true]{biblatex}
\AtEveryBibitem{
\clearfield{issn}
  \clearfield{isbn}
  \clearlist{location}
  \clearfield{month}
  \clearfield{day}
  \clearlist{language}
  \clearlist{translator}
  \renewbibmacro{in:}{}
}
\usepackage{subcaption}
\usepackage[shortlabels]{enumitem}
\usepackage{pdflscape}

\usepackage{fullpage}
\usepackage{mathtools}

\usepackage{amsmath}
\usepackage{amsthm}
\usepackage{thmtools}
\usepackage{hyperref}

\usepackage{tikz}
\usetikzlibrary{calc, backgrounds, shapes.geometric}
\usepackage{amssymb, amsfonts}

\usepackage{tikz-cd}
\usepackage{enumitem}

\usepackage{mathdots}

\usepackage{url}

\usepackage{wrapfig}

\usepackage[T1]{fontenc}

\usepackage{floatflt}
\usepackage{etoolbox}

\theoremstyle{plain}
\newtheorem{theorem}{Theorem}[section]
\newtheorem{lemma}[theorem]{Lemma}
\newtheorem{proposition}[theorem]{Proposition}
\newtheorem{corollary}[theorem]{Corollary}

\theoremstyle{definition}

\newtheorem{ex}[theorem]{Example}
\theoremstyle{remark}
\newtheorem{remark}[theorem]{Remark}

\usepackage{cleveref}

\newcounter{proofstep}
\newcommand{\step}{
  \refstepcounter{proofstep}
  \medskip
  \textit{Step \arabic{proofstep}.}
}
\AtBeginEnvironment{proof}{\setcounter{proofstep}{0}}

\newcounter{paragraphcount}[section]
\newcommand{\titledparagraph}[1]
{\bigskip\noindent%
\refstepcounter{paragraphcount}%
\textbf{\thesection.\theparagraphcount.\ #1}}

\newcommand{\titledparagraphnonr}[1]
{\bigskip\noindent%
\textbf{#1}}

\newcommand{\C}{\mathbb{C}}
\newcommand{\R}{\mathbb{R}}
\newcommand{\Q}{\mathbb{Q}}
\newcommand{\Z}{\mathbb{Z}}
\newcommand{\p}{\mathbb{P}}
\newcommand{\id}{\mathrm{id}}
\newcommand{\trans}{\mathrm{Trans}}
\newcommand{\triv}{\mathrm{Triv}}
\newcommand{\diff}{\mathrm{Diff}}

\newcommand{\MW}{\mathrm{MW}}
\newcommand{\NS}{\mathrm{NS}}
\newcommand{\rk}{\mathrm{rank}}
\newcommand{\MWH}{\mathrm{MW}^\mathrm{hol}}
\newcommand{\PD}{\mathrm{PD}}
\newcommand{\disc}{\mathrm{disc}}
\newcommand{\nr}{\mathrm{nr}}

\title{The Smooth Narrow Mordell--Weil Group of Elliptically Fibered 4-Manifolds}
\author{Maria Morariu}

\begin{document}

\maketitle

\begin{abstract}
For an elliptically fibered $4$-manifold $\pi\colon M\to B$, we introduce a subgroup of the smooth mapping class group of $M$ with explicit geometric representatives: the smooth narrow Mordell--Weil group $\MW_0(\pi)$. 
Leveraging Kodaira's classification of singular fibers, we construct a natural isomorphism between $\MW_0(\pi)$ and the orthogonal complement of the trivial lattice in $H^2(M;\Z)$.
This identification parallels its holomorphic counterpart. This paper generalizes work of Farb and Looijenga \cite{FarbLoojenga_MordellWeil}, who defined the smooth Mordell--Weil group and computed it for fibrations over a sphere with only nodal fibers. We give an explicit formula for the rank of $\MW_0(\pi)$ depending only on the genus of the base curve and the types of singular fibers of $\pi$. As a corollary, we show that rational elliptic surfaces are the only elliptic surfaces with holomorphic and smooth Mordell--Weil groups of the same rank.
\end{abstract}

\section{Introduction}

\titledparagraph{Elliptic fibrations.} 
An \textit{elliptic fibration} of a smooth, projective complex surface $M$ is a proper, surjective, holomorphic\footnote{The assumptions can be weakened as follows: Let $M$ be a closed smooth $4$-manifold and $\pi\colon M\to B$ a proper submersion with connected fibers such that the general fiber is of genus one. Further assume that around each critical value, there exist complex charts on $M$ and $B$ such that $\pi$ is holomorphic in these charts. This ensures that Kodaira's classification of singular fibers of an elliptic fibration still holds (see \Cref{section: Kodaira classification}) and together with the assumptions \ref{item: assumption 1}-\ref{item: assumption 3} is sufficient for \Cref{thm: MW0 as perp of triv}.} 
map $\pi\colon M\to B$ with connected fibers, where $B$ is a smooth, connected projective complex curve, such that the general fiber is a smooth curve of genus $1$. 
Throughout the paper, we make the following assumptions:
\begin{enumerate}[(1)] \setlength{\itemsep}{0pt}
    \item \label{item: assumption 1} There exists a distinguished smooth section $\sigma_0\colon B\to M$, called the \textit{zero-section}.
    \item \label{item: assumption 2} $\pi$ is \textit{nontrivial}:  there exists at least one singular fiber (this ensures $\chi(M)>0$).
    \item \label{item: assumption 3} $\pi$ is \textit{relatively minimal}: no fiber contains an irreducible component of self-intersection $-1$.
\end{enumerate}

The choice of zero-section endows the general fiber with a distinguished point and therefore with an abelian group structure. This additive structure extends to the smooth locus of the singular fibers, see \Cref{section: group structure fibers}.
Given a smooth section $\sigma\colon B\to M$, there exists an associated diffeomorphism $f_\sigma\colon M\to M$ that acts on each fiber by translation by $\sigma$ in the group law of that fiber, see \Cref{section: def and first props}.

\titledparagraph{The smooth (narrow) Mordell--Weil group.}
Farb and Looijenga \cite{FarbLoojenga_MordellWeil} defined the \textit{smooth Mordell--Weil group}:
\[
    \MW(\pi)\coloneqq \pi_0(\{f\in \diff(M)\mid f \text{ acts on each smooth fiber by translation}\}).
\]

This generalizes the classical notion of Mordell--Weil group, $\MWH(\pi)$, obtained by replacing $\diff(M)$ with the group of biholomorphic automorphisms of $M$ in the definition. In this paper, we refer to the smooth Mordell--Weil group as ``the Mordell--Weil group''.

Farb and Looijenga proved that $\MW(\pi)$ is finitely generated. Assuming that $M$ is simply-connected and  all singular fibers are integral (nodes or cusps), they computed $\MW(\pi)$. If arbitrary singular fibers of Kodaira type are allowed, the story becomes richer. This case is the topic of the present paper.

The main object of this paper is \textit{the narrow Mordell--Weil group}:

\[
\MW_0(\pi) = \pi_0 \left( \left\{ f \in \diff(M) \;\middle|\; 
\begin{array}{@{}l@{}}
  f \text{ acts on each smooth fiber by translation and} \\
  \text{takes each irreducible component of each fiber to itself}
\end{array}
\right\} \right).
\]

$\MW_0(\pi)$ can be rephrased in terms of smooth sections: the set of isotopy classes of smooth sections meeting each fiber in the same irreducible component as $\sigma_0$ is a $\MW_0(\pi)$-torsor (see \Cref{lem: sections give diffeomorphisms}).
As before, this is a smooth analog of the holomorphic narrow Mordell--Weil group. The distinction between $\MW_0(\pi)$ and $\MW(\pi)$ is illustrated in \Cref{fig:sections in narrow MW}.

\begin{figure}
    \centering
    \includegraphics[width=0.35\linewidth]{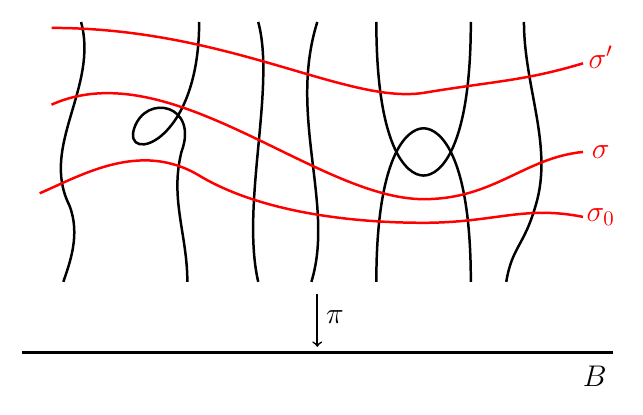}
    \caption{Local picture of an elliptic fibration with three sections. The diffeomorphism that acts by translation by $\sigma$ yields an element of $\MW_0(\pi)$, but the diffeomorphism that acts by translation by $\sigma'$ does not.}
    \label{fig:sections in narrow MW}
\end{figure}

\titledparagraph{Relation between different versions of Mordell--Weil groups.} \label{subsection: relation different versions MW}
This is summarized in the following diagram:
\[
    \begin{tikzcd}
        \MWH_0(\pi) \arrow[d,hook]\arrow[r,hook, "\text{f.i.}"] & \MWH(\pi)\arrow[d,hook]\\
        \MW_0(\pi) \arrow[r,hook,"\text{f.i.}"] & \MW(\pi)\\
    \end{tikzcd}
\]
As in the holomorphic case, $\MW_0(\pi)$ is a finite-index subgroup of $\MW(\pi)$ (see \Cref{prop: MW0 is subgroup of MW} and \Cref{lem: narrow MW is fin index in MW}). The vertical maps are inclusions (\Cref{lem: hol MW is subgroup of smooth}). In general, $\MWH(\pi)$ can contain torsion (see \cite[§6.6]{Schutt_Schioda_MordellWeil}), which implies that $\MW(\pi)$ can contain torsion. $\MW_0(\pi)$ is torsion-free (\Cref{cor: rk MW_0 thm}), so in general $\MW_0(\pi) \lneq \MW(\pi)$. If all fibers are integral (nodes or cusps) then $\MW_0(\pi)=\MW(\pi)$.

\titledparagraph{Main Results.}
Let $\Sigma_g$ denote a projective curve of genus $g\ge 0$.
For an elliptic fibration $\pi\colon M\to \Sigma_g$,
define the \textit{trivial lattice} $\triv(\pi)\subset H^2(M)\coloneqq H^2(M;\Z)$, as the subgroup generated by the Poincaré duals of the fiber class $[F]\in H_2(M)$, the zero-section $[\sigma_0]\in H_2(M)$ and all irreducible components of all singular fibers. The following is our main theorem.
\begin{theorem}[\textbf{Computing $\MW_0(\pi)$}]
\label{thm: MW0 as perp of triv}
    Let $g\ge0$.
    For $\pi\colon M\to \Sigma_g$ an elliptic fibration, there exists a group isomorphism
    \[
        \MW_0(\pi)\cong \triv(\pi)^\perp \subset H^2(M),
    \]
    where $\perp$ is taken in $H^2(M)$ with respect to the intersection form.
\end{theorem}

\begin{remark}
    \Cref{thm: MW0 as perp of triv} underscores the parallel to the holomorphic case. An analogous identity holds for $\MWH_0(\pi)$, but $\perp$ is taken in the Néron-Severi lattice $\NS(M)=H^2(M;\Z)\cap H^{1,1}(M)$, instead of $H^2(M)$.
\end{remark}

\begin{remark}
    \Cref{thm: MW0 as perp of triv} reduces the computation of $\MW_0(\pi)$ to determining the singular fibers of an elliptic fibration and intersection theory in $H^2(M)$.
\end{remark}

The singular fibers of an elliptic surface are of two types: \textit{additive}, if they are simply-connected, or \textit{multiplicative}, otherwise. 
\Cref{fig:diff ell fib with MW_0 of rk 4} illustrates three examples.

\begin{corollary} \label{cor: rk MW_0 thm}
    Let $\pi\colon M\to \Sigma_g$ be an elliptic surface with $a\ge 0$ additive singular fibers and $m\ge 0$ multiplicative ones.
    Then
    \[
        \MW_0(\pi)\cong \Z^{4g+2a+m-4}.
    \]
\end{corollary}

\begin{remark}
    The rank of $\MW_0(\pi)$ does not depend on how the singular fibers are embedded in $M$. This is in stark contrast to the holomorphic case. Indeed, the rank of $\MW_\text{hol}(\pi)$ depends on the rank of the Néron-Severi group, 
    see the Shioda--Tate formula (\Cref{cor: Shioda Tate Formula}).
\end{remark}

\begin{corollary} \label{cor: rk MW}
    Let $\pi\colon M\to \Sigma_g$ be an elliptic surface with $a\ge0$ additive singular fibers and $m\ge0$ multiplicative ones.
    Then $\MW(\pi)$ is finitely generated and the rank of its free part is $4g+2a+m-4$.
\end{corollary}
\begin{proof}
    By \Cref{lem: narrow MW is fin index in MW}, $\MW_0(\pi)$ is a finite-index subgroup of $\MW(\pi)$, so they have equal rank. Since $\MW_0(\pi)$ is finitely generated, $\MW(\pi)$ is finitely generated as well.
\end{proof}

We can compare the holomorphic and smooth Mordell--Weil groups as follows:
\begin{theorem}[\textbf{Holomorphic vs. smooth}] \label{thm: intro smooth vs hol}
    If $\pi\colon M\to \Sigma_g$ is an elliptic surface of arithmetic genus $d=\chi(M)/12$, then
    \[
        \rk(\MW(\pi))-\rk(\MWH(\pi))\ge 2(g+d-1).
    \]
\end{theorem}

\begin{remark} \label{rmk: rk MW(K3) ge 2}
    Since $\rk(\MWH(\pi))$ is nonnegative, \Cref{thm: intro smooth vs hol} implies the rank of the smooth Mordell--Weil group is at least $2(g+d-1)$. In particular, if $M$ is a K3 surface ($g=0,d=2$), then $\rk(\MW(\pi))\ge 2$. 
\end{remark}

\begin{remark}
    To the knowledge of the author, the highest known rank of $\MW_\text{hol}(\pi)$ for an elliptic fibration over $\C(t)$ is $68$, attained for a highly symmetric elliptic surface with $360$ cuspidal singularities (see the Remark on page 110 in \cite{Shioda_RmksEllCurves_FunctionFields}).
    In contrast, $\rk(\MW(\pi))\to \infty$ as $d\to \infty$.
\end{remark}

A consequence of \Cref{thm: intro smooth vs hol} is that the smooth Mordell--Weil group has strictly larger rank than the holomorphic one, unless the surface is rational elliptic: 

\begin{corollary}[\textbf{Characterizing rational elliptic surfaces}]
\label{cor: rank MW strictly bigger than MWH unless rational}
Let $\pi\colon M \to B$ be an elliptic surface. Then
    \[
        \rk(\MWH(\pi))\le \rk(\MW(\pi)),
    \]
    with equality if and only if $M$ is a rational elliptic surface. 
\end{corollary}

\noindent\textbf{Methods.} The key to our computation is a cohomological characterization for $\MW_0(\pi)$ (\Cref{prop: MW0=H1}). We use a spectral sequence argument, leveraging Kodaira's classification of singular fibers.

\begin{figure}
    \centering
    \begin{subfigure}[b]{0.265\textwidth}
        \centering
        \fbox{\includegraphics[width=\textwidth]{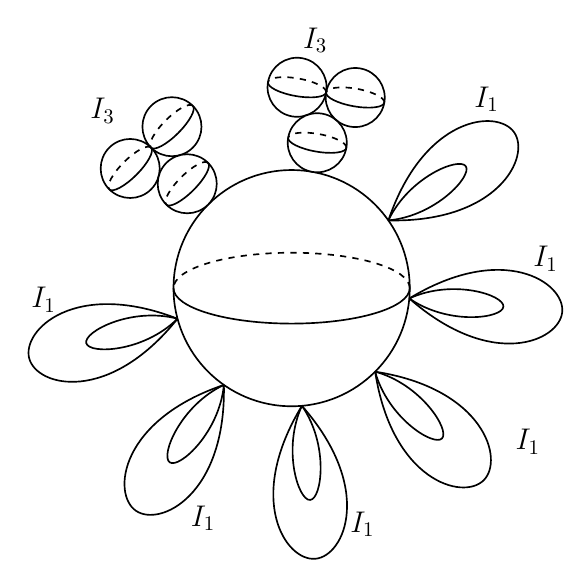}}
        \caption{Rational elliptic surface with singular fibers $6I_1+2I_3$.}
        \label{fig:2I_3}
    \end{subfigure}
    \hfill
    \begin{subfigure}[b]{0.276\textwidth}
        \centering
        \fbox{\includegraphics[width=\textwidth]{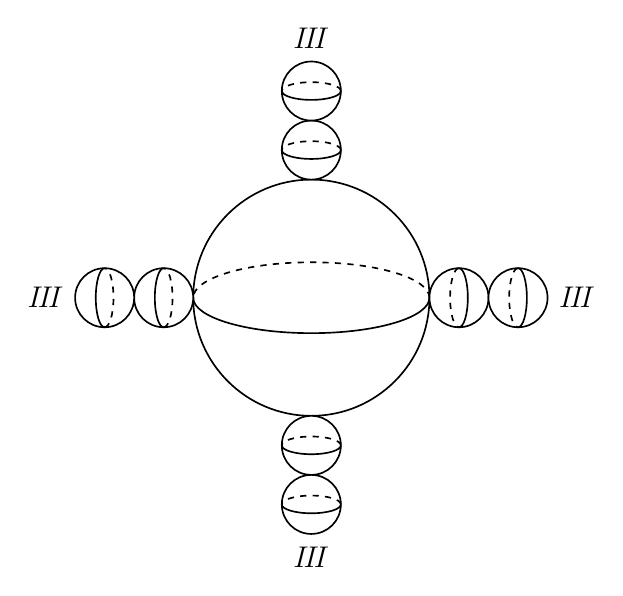}}
        \caption{Rational elliptic surface with singular fibers $4III$.}
        \label{fig:4III}
    \end{subfigure}
    \hfill
    \begin{subfigure}[b]{0.265\textwidth}
        \centering
        \fbox{\includegraphics[width=\textwidth]{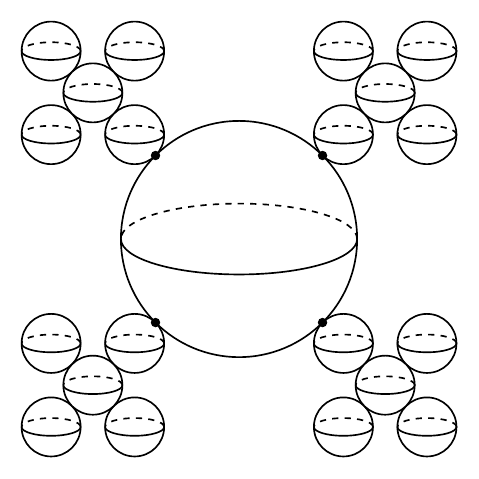}}
        \caption{Kummer K3 surface with singular fibers $4I_0^*$.}
        \label{fig:Kummer}
    \end{subfigure}
    \caption[Each of these elliptic fibrations $\pi\colon M\to \p^1$ has $\MW_0(\pi)\cong \Z^4$.]{Each of these elliptic fibrations\footnotemark $\pi\colon M\to \p^1$ has $\MW_0(\pi)\cong \Z^4$.} 
    \label{fig:diff ell fib with MW_0 of rk 4}
\end{figure}
\footnotetext{See \Cref{section: Kodaira classification} for terminology describing the singular fibers.}

\begin{ex}[\textbf{Rank of $\MW$ can distinguish elliptic fibrations on a fixed diffeomorphism type}]
    As Kodaira showed, all K3 surfaces are diffeomorphic. However, the Mordell--Weil group can differentiate between  elliptically fibered K3 surfaces that are not isomorphic as elliptic fibrations (see \cite[Definition 3.2]{FriedmanMorgan_Smooth4Mflds}). By \Cref{rmk: rk MW(K3) ge 2}, the rank of $\MW(\pi)$ for an elliptic fibration $\pi\colon M\to \p^1$ of a K3 surface $M$ is at least $2$. By \Cref{cor: rk MW}, this is attained if $\pi$ has $6$ multiplicative fibers. See \cite{MirandaPersson_ConfigurationEllipticK3} for the $112$ possible configurations of $6$ multiplicative fibers on an elliptically fibered K3 surface. The other extreme is $\rk(\MW(\pi))=20$, obtained for the generic elliptically fibered K3 surface with $24$ nodal singular fibers.
\end{ex}

\titledparagraph{$\MW_0(\pi)$ as a subgroup of the mapping class group.}
$\MW_0(\pi)$ acts faithfully by Eichler transformations on $H^2(M)$ (see \Cref{prop: narrow MW as annihilator}). This action factors as
\[
    \MW_0(\pi) \to \pi_0(\diff(M)) \to \mathrm{Aut}(H^2(M),\smile).
\]
Injectivity of the composition implies that $\MW_0(\pi)$ injects into the smooth mapping class group $\pi_0(\diff(M))$.
$\MW_0(\pi)$ plays an important role in the classification of mapping classes of elliptic fibrations, representing the unipotent radical in a Thurston-type classification for elliptically fibered four-manifolds.
See \cite[§1.4]{FarbLoojenga_MordellWeil} for a discussion in the case that the singular fibers are integral.

\titledparagraph{Related Work.} Under the assumption that the $j$-invariant is nonconstant, Shioda \cite{Shioda_ModularEllSurf} computed
the rank of $H^1(B,R^1\pi_*\Z)$. He determined the same formula as the rank of the narrow smooth Mordell--Weil group in \Cref{cor: rk MW_0 thm}.
Hacking--Keating \cite[Section 4.1]{HackingKeating_Symplectomorphisms} computed a symplectic version of the Mordell--Weil group. More precisely, they considered an elliptic surface over a disk with an exact symplectic form such that the fibers are Lagrangian and all singular fibers are nodal. They computed the group of isotopy classes of symplectomorphisms that act by translation in the fibers.
Cox and Zucker \cite{CoxZucker_IntersectionNumbers} embedded the holomorphic Mordell--Weil group in $H^1(B,R^1\pi_*\Q)$ and used the Leray spectral sequence to compute the latter.

\titledparagraph{Future Directions.}
A consequence of \Cref{cor: rk MW_0 thm} is that $\MW_0(\pi)$ is torsion-free. 
Thus, the question remains: what is the torsion subgroup of $\MW(\pi)$? Can it be larger than that of the holomorphic Mordell--Weil group, or does a version of Nielsen realization hold for Mordell--Weil elements?
We tackle these questions, as well as a full computation of $\MW(\pi)$, in a subsequent paper.

\titledparagraph{Outline.} In \Cref{section: background}, we synthesize the facts about elliptic fibrations that we use in the rest of the paper. In \Cref{section: def and first props}, we prove some key properties of the narrow Mordell--Weil group. The cohomological characterization given in \Cref{subsection: cohomological characterization} is essential for our computation of $\MW_0(\pi)$. 
In \Cref{subsection: MW0 as annihilator}, we express $\MW_0(\pi)$ as the quotient of an intersection of annihilators.
This is used in \Cref{subsection: proof main thm} to prove \Cref{thm: MW0 as perp of triv}. In \Cref{subsection: rank computation of MW0}, we compute the rank of $\MW_0(\pi)$. In \Cref{subsection: comparing smooth and hol}, we prove \Cref{thm: intro smooth vs hol} and \Cref{cor: rank MW strictly bigger than MWH unless rational}. In \Cref{subsection: examples}, we collect some examples.

\titledparagraphnonr{Computational Resource Disclosure.} Gemini Pro 3.1 was used in the creation of pictures. ChatGPT Pro 5.6 was involved in checking the proofs.

\titledparagraphnonr{Acknowledgements.} I am very grateful to my advisor Benson Farb for suggesting this project, for his constant support, and for extensive comments on earlier drafts. I am thankful to Eduard Looijenga for many useful discussions. I am thankful to Dan Margalit for helpful comments on an earlier draft.

\section{Background on Elliptic Fibrations}
\label{section: background}
In this section, we collect some background, which will be used throughout the paper.

\medskip\noindent \textbf{Notation and Conventions:} For an elliptic fibration $\pi\colon M\to B$ and $b\in B$, we denote $M_b\coloneqq \pi^{-1}(b)$. 
We denote by $R\subset B$ the set of critical values of $\pi$.
Throughout the paper, (co)homology is taken with $\Z$-coefficients unless otherwise specified. The topology on $\diff(M)$ is the Whitney $C^\infty$ topology. From now on, we assume that the zero-section $\sigma_0$ is holomorphic in a neighborhood of each critical value\footnote{This assumption is not restrictive, since any smooth section can be deformed into a section which is holomorphic in a neighborhood of each critical value by the holomorphic submersion theorem.}.

For any elliptic surface over a projective curve $\Sigma_g$, the Euler characteristic is a multiple of 12\footnote{This is a consequence of Noether's Theorem and the fact that the canonical class is a multiple of the fiber class}.
The non-triviality assumptions \ref{item: assumption 1}-\ref{item: assumption 3} about elliptic surfaces ensure $\chi(M)>0$.
A fundamental invariant of an elliptic surface is its arithmetic genus, $d\coloneqq \chi(M)/12$. Together with $g$, this determines how $M$ fits in the Enriques-Kodaira classification of complex surfaces: 
\begin{itemize}
    \item If $d=1,g=0$, $M$ is a \textit{rational elliptic surface} (Kodaira dimension $-\infty$);
    \item if $d=2,g=0$, $M$ is an \textit{elliptically fibered K3 surface} (Kodaira dimension $0$);
    \item if $d\ge 3$ or $g\ge1$, $M$ has Kodaira dimension $1$.
\end{itemize}
The following proposition shows that $\pi_1(M)\cong \pi_1(B)$.
\begin{proposition}[\protect{\cite[p. 157, Proposition 2.1]{FriedmanMorgan_Smooth4Mflds}}] \label{prop: fund group}
    If $\pi\colon M\to \Sigma_g$ is an elliptic fibration, then $\pi$ induces an isomorphism $\pi_*\colon \pi_1(M)\to \pi_1(\Sigma_g)$.\footnote{The cited theorem assumes that $\pi$ is non-trivial and without multiple fibers. These assumptions are automatically satisfied in our context, by assumptions \ref{item: assumption 1}-\ref{item: assumption 3}}
\end{proposition}

\begin{corollary}
    If $\pi\colon M\to \Sigma_g$ is an elliptic fibration, then $H^2(M)$ is torsion-free. Together with the intersection form, it is a unimodular lattice.
\end{corollary}
\begin{proof}
    Applying abelianization and the Hurewicz isomorphism to both sides of $\pi_1(M)\cong \pi_1(\Sigma_g)$ from \Cref{prop: fund group} implies $H_1(M)\cong \Z^{2g}$. By the universal coefficient theorem, $H^2(M)$ is torsion-free. Poincaré duality implies unimodularity of the lattice $H^2(M)$.
\end{proof}

\begin{lemma}\label{lem: euler char is sum of euler of fibers}
    For an elliptic surface $\pi\colon M \to B$ with set of critical values $R\subset B$,
    \[
        \chi(M)= \sum_{t\in R} \chi(M_{t}).
    \]
\end{lemma}
\begin{proof} By \cite[Lemma VI.4]{Beauville_ComplexAlgSurfaces},
    \[
        \chi(M)=\chi(B)\chi(T^2) + \sum\limits_{b\in R} (\chi(M_b)-\chi(T^2)).
    \]
    The Euler characteristic of a torus is $0$, yielding the conclusion.
\end{proof}

\subsection{Kodaira's Classification of Singular Fibers} \label{section: Kodaira classification}

Let $\pi\colon M\to B$ be an elliptic surface and let $M_b$ be a singular fiber. We can write it as a divisor $M_b=\sum m_i C_i$, where $C_i$ are the distinct irreducible components of $M_b$ and $m_i$ is the multiplicity of the component $C_i$. 
Since $\pi\circ\sigma_0=\id_B$, the chain rule implies $d\pi_{\sigma_0(b)}\circ d\sigma_0(b) = \id_{T_b B}$, so $d\pi_{\sigma_0(b)}$ has full rank. Hence, $\sigma_0(b)$ lies in a multiplicity-one component of $M_b$.
In particular, there exists $i$ such that $m_i=1$.\footnote{The assumption that a section exists is essential, otherwise $\pi$ could have a multiple fiber.}
The \textit{dual graph} to a reducible fiber $M_b$ is the graph with vertices corresponding to $C_i$ and exactly $C_i\cdot C_j$ edges connecting the vertices corresponding to $C_i$ and $C_j$, if $i\ne j$.

The singular fibers of an elliptic surface were first classified by Kodaira \cite{Kodaira_CompAnalyticSurf2} and subsequently by Néron \cite{Neron_ModelMin} and Tate \cite{Tate_AlgorithmSingFibers}. The formulation of the classification we give below can be found in \cite[Theorem 5.12, Proposition 5.15]{Schutt_Schioda_MordellWeil}.

\begin{theorem}[Kodaira Classification]\label{thm: Kodaira classification sing fibers}
    The dual graph of a reducible fiber of an elliptic surface is an extended Dynkin diagram of type ADE (see \Cref{fig: extended Dynkin ADE}).
    Furthermore, all fibers of an elliptic surface have one of the following types:
    \begin{center}
    \begin{tabular}{c|c|c}
       \textbf{Kodaira Symbol}  & \textbf{Description} & \textbf{Dual Graph} \\ \hline
       $I_0$  & smooth elliptic curve & \\
       $I_1$ & nodal rational curve & \\
       $I_N, N\ge 2$ & $N$ smooth rational curves meeting transversally in a cycle & $\tilde{A}_{N-1}$\\
       $I_N^*, N\ge 0$ & $N+5$ smooth rational curves & $\tilde{D}_{N+4}$ \\
       $II$ & a cuspidal rational curve &\\
       $III$ & two smooth rational curves meeting at one point with multiplicity $2$ & $\tilde{A}_{1}$\\
       $IV$ & three smooth rational curves meeting at one point & $\tilde{A}_{2}$\\
       $IV^*$ & $7$ smooth rational curves& $\tilde{E}_6$ \\
       $III^*$ & $8$ smooth rational curves & $\tilde{E}_7$ \\
       $II^*$ & $9$ smooth rational curves &  $\tilde{E}_8$
    \end{tabular}
    \end{center}
\end{theorem}

\begin{figure}
    \centering
    \includegraphics[width=0.85\linewidth]{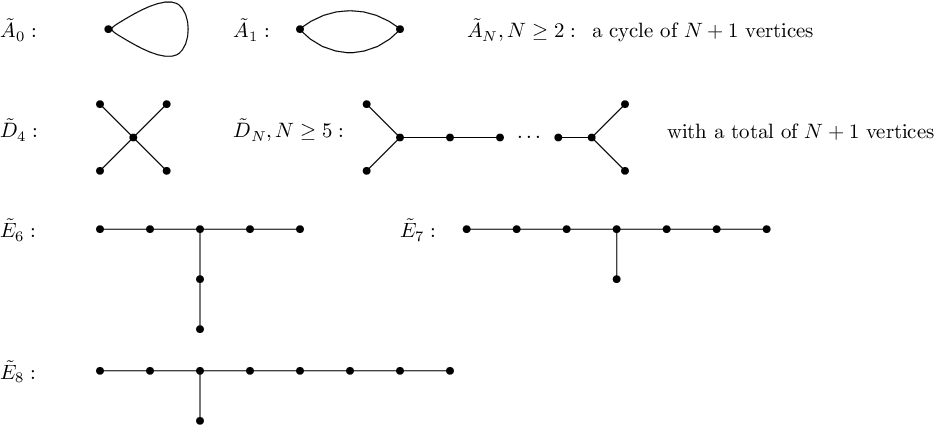}
    \caption{Extended Dynkin Diagrams of type ADE}
    \label{fig: extended Dynkin ADE}
\end{figure}

Consider a reducible fiber $M_b$ with irreducible components $C_0,\dots,C_r$, where $C_0$ is the unique component of $M_b$ intersecting $\sigma_0$. The negative \textit{extended Cartan matrix} $\tilde A$ associated to $M_b$ is the $(r+1)\times (r+1)$ matrix with entries $\tilde A_{i,j}=C_i\cdot C_j$. 
Furthermore, we can construct a (non-extended) Dynkin diagram of type ADE, by removing the vertex corresponding to $C_0$ from the dual graph of $M_b$. There is a naturally associated negative \textit{Cartan matrix} $A$, an $r\times r$ matrix such that
\[
    \tilde A=\begin{pmatrix}
        -2 & v^T \\
        v & A
    \end{pmatrix}.
\]
The matrix $A$ has full rank and $|\det A|$ is equal to the number of components of multiplicity $1$ of $M_b$.

\subsection{Group Structure on Fibers}\label{section: group structure fibers}
We explain how the group structure on the smooth fibers extends to the smooth locus of the singular fibers.
Consider two complex manifolds $\mathcal{G},\Delta$ with a holomorphic submersion $\Psi\colon \mathcal{G}\to \Delta$ and a holomorphic section $e\colon \Delta\to \mathcal{G}$. Assume that for all $t\in \Delta$, the preimage $\Psi^{-1}(t)$ is a Lie group with identity element $e(t)$, whose complex structure is that of a complex submanifold of $\mathcal{G}$. Denote
\[
    \mathcal{D}\coloneqq \{(g_1,g_2)\in \mathcal{G}\times \mathcal{G}: \Psi(g_1)=\Psi(g_2)\}\subset \mathcal{G}\times \mathcal{G}.
\]
$\mathcal{G}$ is an \textit{analytic fiber system of groups} over $\Delta$ if the map 
\[
    \mathcal{D}\to \mathcal{G}, (g_1,g_2)\mapsto g_1g_2^{-1}
\]
is holomorphic.

Let $\pi\colon M\to \Delta$ be an elliptic fibration over the unit disk with only one singular fiber at the origin, such that the zero-section $\sigma_0$ is holomorphic\footnote{Recall that we are assuming that $\sigma_0$ is holomorphic in a neighborhood of singular fibers, so this can be attained by shrinking the codomain.}.
Denote $\Delta'\coloneqq \Delta-\{0\}$ and $M'\coloneqq \pi^{-1}(\Delta')$. Then
$\pi\colon M'\to \Delta'$ is an analytic fiber system of groups over $\Delta'$ with the holomorphic section $e$.

Denote by $M^\#\coloneqq \{x\in M: d\pi_x \text{ has full rank}\}$ and let $M_0^\#\coloneqq M_0\cap M^\#$ be the smooth locus of $M_0$. This means $M^\#=M'\cup M_0^\#$. Furthermore, $M_0^\#$ is the union of the smooth loci of the irreducible components of $M_0$ which have multiplicity $1$.

\begin{theorem}[\protect{\cite[Theorem 9.1]{Kodaira_CompAnalyticSurf2}}]\label{thm: group on sing fibers}
    There exists on $M^\#$ a unique structure of analytic fiber system of abelian groups over $\Delta$, which is an extension of the structure on $M'$ of analytic fiber system of abelian groups over $\Delta'$. The Lie group structure of $M_0^\#$, depending on the Kodaira type of $M_0$ is given below:
    \begin{center}
    \renewcommand{\arraystretch}{1.3}
        \begin{tabular}{c|c}
        \textbf{Type} & \textbf{Group Structure} \\ \hline
        $I_n,n\ge 1$ & $\C^*\times \Z/n\Z$ \\
        $I_{2n}^*$ & $\C\times (\Z/2\Z)^2$\\
        $I_{2n+1}^*$ & $\C\times \Z/4\Z$\\
        $II$ & $\C$\\
        $III$ & $\C\times \Z/2\Z$\\
        $IV$ & $\C\times \Z/3\Z$\\
        $IV^*$ & $\C\times \Z/3\Z$\\
        $III^*$ & $\C\times \Z/2\Z$\\
        $II^*$ & $\C$
    \end{tabular}
    \end{center}
\end{theorem}

If we denote by $C_0^\#$ the smooth part of the zero-component of $M_0$, there is the following short exact sequence of topological groups:
\[
    0\to C_0^\#\to M_0^\# \to \pi_0(M_0^\#)\to 0.
\]

A singular fiber $M_0$ of an elliptic surface with irreducible component $C_0$ intersecting the zero-section is of one of the following types: \textit{multiplicative}, if $C_0^\#\cong \C^*$, and \textit{additive}, if $C_0^\#\cong \C$. The only multiplicative fibers are those of type $I_n$, while all others are additive. Furthermore, a singular fiber is additive if and only if it is simply-connected.

\begin{lemma}\label{lem: euler char of sing fibers}
    Let $M_0$ be a singular fiber of an elliptic surface with $\mathrm{nr}(M_0)$ irreducible components.
    \begin{itemize}
        \item If $M_0$ is multiplicative, then $\chi(M_0)=\mathrm{nr}(M_0)$.
        \item If $M_0$ is additive, then $\chi(M_0)=\mathrm{nr}(M_0)+1$.
    \end{itemize}
\end{lemma}
\begin{proof}
    This follows directly by considering each case in Kodaira's classification and using inclusion-exclusion for the underlying topological space.
\end{proof}

\subsection{The holomorphic (Narrow) Mordell-Weil Group}
The \textit{holomorphic Mordell-Weil group} of an elliptic fibration $\pi\colon M\to B$ is 
\[
    \MWH(\pi)=\{f\in \mathrm{Aut}(M): f \text{ acts on each smooth fiber by translation}\},
\]
where $\mathrm{Aut}(M)$ denotes the group of biholomorphic automorphisms of $M$. $\MWH(\pi)$ can be expressed in terms of the following two lattices (see \cite[Theorem 6.5]{Schutt_Schioda_MordellWeil}):
\begin{itemize}
    \item The \textit{Néron-Severi group} $\NS(M)$ is the group of divisors of $M$ modulo algebraic equivalence. The Lefschetz (1,1)-theorem implies that
    \begin{equation}\label{eq: Lefschetz 1,1 for NS}
    \NS(M)=H^2(M;\Z)\cap H^{1,1}(M).
    \end{equation}
    \item The \textit{trivial lattice} $\triv(\pi)\subset \NS(M)$ is the lattice generated by the zero-section and the fiber components.\footnote{This is equivalent to our earlier definition of the trivial lattice, since the class of the generic fiber is equal in homology to a weighted sum of the irreducible components of a singular fiber.}
\end{itemize}

\begin{proposition}[Shioda-Tate Formula, \protect{\cite[Corollary 6.7]{Schutt_Schioda_MordellWeil}}]\label{cor: Shioda Tate Formula}
    Let $\pi\colon M\to B$ be an elliptic fibration with set of critical values $R\subset B$. For each $b\in R$, let $\nr(b)$ denote the number of components of the singular $M_b$. Then
    \[
    \mathrm{rank}(\NS(M))=\mathrm{rank}(\MWH(\pi))+2+\sum_{b\in R} (\nr(b)-1).
    \]
\end{proposition}

\noindent The \textit{holomorphic narrow Mordell-Weil group} is 
\[
    \MWH_0(\pi)=\{f\in \MW^\text{hol}(\pi): f \text{ takes each irreducible component of each singular fiber to itself}\}.
\]
$\MWH_0$ is a torsion-free subgroup of finite index in $\MW^\text{hol}$.

\begin{proposition}[\protect{\cite[Theorem 6.47]{Schutt_Schioda_MordellWeil}}] 
\label{prop: MW_0 hol is perp of triv}
    For an elliptic fibration $\pi\colon M\to B$, there is a natural group isomorphism 
    \[
        \MWH_0(\pi)\cong \triv(\pi)^\perp \subset \NS(M),
    \]
    where $\perp$ is taken with respect to the intersection form in $\NS(M)\subset H^2(M)$.
\end{proposition}

\section{First Properties of the Narrow Mordell-Weil Group}
\label{section: def and first props}
Let $\pi\colon M\to B$ be an elliptic fibration. 
Define
\[
\trans(\pi)\coloneqq \{f\in \diff(M)\mid f \text{ acts on each smooth fiber by translation}\}.
\]
For each $f\in \trans(\pi)$, there exists a naturally associated section $\sigma_f=f\circ\sigma_0$.
Further define
\[
    \trans_0(\pi)\coloneqq \{f\in \trans(\pi)\mid f \text{ takes each irreducible component of each fiber to itself} \}.
\]
If $f\in \trans_0(\pi)$, the associated section $\sigma_f$ intersects each singular fiber in the same reducible component as the zero-section.
We can express the (narrow) Mordell-Weil group in terms of the translation groups: $\MW(\pi) = \pi_0(\trans(\pi))$ and $\MW_0(\pi) = \pi_0(\trans_0(\pi))$.

The following lemma shows that a smooth section of $\pi$ naturally defines a diffeomorphism of $M$. 
\begin{lemma}[\textbf{Sections define diffeomorphisms}]\label{lem: sections give diffeomorphisms}
    Consider an elliptic fibration $\pi\colon M\to B$ over a disk $B$ with only one critical value at the origin.
    For any smooth section $\sigma\colon B\to M$ there exists a unique diffeomorphism $f_\sigma$ of $M$ that acts by translation by $\sigma$ on each smooth fiber of $M$ in the group law with identity given by $\sigma_0$.
\end{lemma}
\begin{proof}
    The definition of a section implies $\pi\circ \sigma_0=\id_B$, so $d\pi_{\sigma_0(0)}d(\sigma_0)_0=\id_{T_0B}$. Thus, $\pi\colon M\to B$ is a holomorphic submersion at $\sigma_0(0)$, so there exists a neighborhood $U\subset B$ of the origin and a holomorphic section $e\colon U\to M_U\coloneqq \pi^{-1}(U)$ with $e(0)=\sigma_0(0)$.
    The smooth locus $M_U^\#$ is the analytic Néron model with respect to $e$ and it acts holomorphically on the entire $M_U$. This means there is a holomorphic action 
    \[
        \alpha\colon M_U^\#\times_U M_U\to M_U,
    \]
    extending the self-action by translation of $M_U^\#$ (see \cite[Lemma 10.2.12(c)]{Liu_AG_ArithmCurves}). 
    Since the image of any section is contained in the smooth locus of $\pi$, we can define the smooth section 
    \[
        s\coloneqq \sigma|_U-_e\sigma_0|_U \colon U\to M_U^\#,
    \]
    where $-_e$ is taken with respect to the group law defined by $e$ on the smooth locus of each fiber.
    Define $f_U\colon M_U\to M_U$ by $f_U(x)=\alpha(s(\pi(x)),x)$. Since $\alpha$ is holomorphic and $\sigma,\sigma_0$ are smooth, $f_U$ is a smooth map. The map $g_U\colon M_U\to M_U$, given by $g_U(x)=\alpha(-_e s(\pi(x)),x)$ is its inverse, so $f_U$ is a diffeomorphism of $M_U$. 
    
    Define $F\colon M-M_0\to M-M_0$ as the diffeomorphism 
    $F(x)=x+_{\sigma_0(\pi(x))} \sigma(\pi(x))$. Its inverse is translation by the inverse of $\sigma(\pi(x))$ in the group law of $\sigma_0(\pi(x))$.
    For $b\in U-\{0\}$, the group laws with respect to $\sigma_0(b)$ and $e(b)$ satisfy the following relation:
    \[
        \forall x\in M_b: x+_e\sigma(b)-_e \sigma_0(b) = x+_{\sigma_0(b)} \sigma(b).
    \]
    Thus, $F$ and $f_U$ coincide where their domains overlap and so do their inverses. Thus, they glue to a diffeomorphism $f_\sigma$ of $M$.

    Finally, any diffeomorphism $f$ that acts by translation by $\sigma$ on the smooth fibers of $M$ coincides with $f_\sigma$ on $M-M_0$. Since this set is dense, continuity implies that $f=f_\sigma$.
\end{proof}

\subsection{Relation between Different Versions of Mordell-Weil Groups}
We show that the different versions of Mordell-Weil groups relate to each other as one would expect.
\begin{lemma} \label{lem: hol MW is subgroup of smooth}
Let $\pi\colon M\to B$ be an elliptic fibration.
    The natural inclusion $\MWH(\pi)\xhookrightarrow{}\trans(\pi)$ induces an injection $\MWH(\pi)\xhookrightarrow{}\MW(\pi)$. Furthermore, the natural inclusion $\MWH_0(\pi)\xhookrightarrow{}\trans_0(\pi)$ induces an injection $\MWH_0(\pi)\xhookrightarrow{}\MW_0(\pi)$.
\end{lemma}
\begin{proof}
    It suffices to show that if two holomorphic sections $\sigma_1,\sigma_2\colon B\to M$ are isotopic then they are equal. Assume there is an isotopy $H\colon B\times [0,1]\to M$ from $\sigma_1$ to $\sigma_2$. Then $\mathrm{im}(H)$ is a singular $3$-chain in $M$ with boundary $\sigma_2(B)-\sigma_1(B)$, so $[\sigma_2(B)]=[\sigma_1(B)]$ in $H_2(M)$. If we further assume that $\sigma_1\ne\sigma_2$ as functions, the fact that they are holomorphic, together with compactness of $B$, implies that $\sigma_1(x)=\sigma_2(x)$ for at most finitely many $x\in B$. 
    For any $x,y\in B$, the equality $\sigma_1(x)=\sigma_2(y)$ implies $x=y$, since $\sigma_1,\sigma_2$ are sections of $\pi$.
    Thus, their images are disjoint or intersect in finitely many points. The fact that $\sigma_1(B)$ and $\sigma_2(B)$ are complex curves implies $[\sigma_1(B)]^2=[\sigma_1(B)]\cdot [\sigma_2(B)] \ge 0$. However, $[\sigma_1(B)]^2=-d$ and non-triviality of $\pi$ implies $d>0$ (see \cite[Corollary 5.45]{Schutt_Schioda_MordellWeil}). This is a contradiction.
    The second assertion follows similarly.
\end{proof}

The following shows that the obvious map $\MW_0(\pi)\to \MW(\pi)$ is injective. This is not clear a priori, since the notion of isotopy is different.
\begin{proposition} \label{prop: MW0 is subgroup of MW}
    Let $\pi\colon M\to B$ be an elliptic fibration.
    If a diffeomorphism $f\in \trans_0(\pi)$
    represents $[\id_M]$ in $\MW(\pi)$, then it also represents $[\id_M]$ in $\MW_0(\pi)$.
\end{proposition}
\begin{proof}
    The assumption that $f$ represents $[\id_M]\in \MW(\pi)$ means that there exists an isotopy $H\colon M\times [0,1]\to M$ from $f$ to $\id_M$, such that $\forall t\in [0,1], H(\cdot,t)\in \trans(\pi)$.
    In particular, $\pi\circ H=\pi$, so each singular fiber and its set of nodes are preserved.
    We show that for all $t\in[0,1]$, it must fix all irreducible components of fibers.
    
    Let $\eta\in B$ be a point such that $M_\eta\coloneqq \pi^{-1}(\eta)$ is reducible.
    Let $H_\eta\colon M_\eta \times [0,1]\to M_\eta$ be the restriction of $H$ to this fiber and denote $C_0^\times,C_1^\times,\dots,C_r^\times$ its irreducible components, with nodes removed.
    Any diffeomorphism that acts by translation in the fibers takes the set $R_\eta$ of nodes of $M_\eta$ to itself.
    Hence, $H_\eta$ restricts to an isotopy $H_\eta\colon M_\eta\setminus R_\eta\times [0,1]\to M_\eta\setminus R_\eta$. For each $t\in[0,1]$, $H_\eta^t\coloneqq H_\eta(\cdot,t)$ is a continuous map.
    Hence, $\forall t\in[0,1]$, $H_\eta^t$ permutes the connected components $C_0^\times,\dots,C_r^\times$, so there is a permutation $\varphi(t)\in S_{r+1}$. 
    Continuity of $H_\eta$ implies continuity of $\varphi\colon [0,1]\to S_{r+1}$. The target is discrete, so this map must be constant.

    Furthermore, $H^0=f\in \trans_0(\pi)$, so $H_\eta^0(C_i^\times)\subseteq C_i^\times$, meaning 
    $\varphi(0)$ is the identity permutation. Hence, $\varphi(t)$ is the identity permutation for all $t\in [0,1]$, so $H_\eta^t$ preserves the irreducible components of $M_\eta$. Since $\eta\in B$ was chosen arbitrarily, this completes the proof.
\end{proof}

\begin{lemma}\label{lem: narrow MW is fin index in MW}
    For any elliptic fibration $\pi\colon M\to B$ over a projective curve $B$, the group $\MW_0(\pi)$ is a finite-index subgroup of $\MW(\pi)$.
\end{lemma}
\begin{proof}
    Since $B$ is a closed curve, the set $R\subset B$ of critical values of $\pi$ is finite. 
    By \Cref{thm: group on sing fibers}, the group structure on the smooth locus $M_t^\#$ of the singular fiber $M_t$ is of the form $\mathbb{G}_t\times F_t$, where $\mathbb{G}_t=\C$ or $\C^\times$ and $F_t=\pi_0(M_t^\#)$ is the group of components of multiplicity $1$ of $M_t$. In particular, $F_t$ is finite.
    
    For $f\in\trans(\pi)$, $f\circ\sigma_0$ is a section of $\pi$. Thus, it must intersect every fiber with multiplicity $1$, so for any $t\in R$, $f(\sigma_0(t))\in M_t^\#$. Also, the map $\trans(\pi)\to M_t^\#$ is a group homomorphism, since any element of $\trans(\pi)$ acts on $M_t^\#$ by translation.
    Hence, there is a natural homomorphism
    \[
    \begin{tikzcd}
        \trans(\pi)\arrow[r]\arrow[rr, bend right=25, "\varphi"'] &\prod\limits_{t\in R} M_t^\# \arrow[r] &\prod\limits_{t\in R} F_t,
    \end{tikzcd}
    \]
    where the first map takes $f\in\trans(\pi)$ to $(f(\sigma_0(t)))_{t\in R}$ and the second is  projection $M_t^\#\to \pi_0(M_t^\#)=F_t$.
    The codomain is discrete, so $\varphi$ factors through a homomorphism $\Bar{\varphi}\colon \MW(\pi)\to \prod_{t\in R} F_t$.
    A translation $f\in \trans(\pi)$ that takes $\sigma_0(t)$ to the identity element in $F_t$ fixes all irreducible components of $M_t$. Thus, $\ker(\varphi)=\trans_0(\pi)$, which implies $\MW_0(\pi)= \ker(\Bar{\varphi})$. Thus, $\MW(\pi)/\MW_0(\pi)\cong \mathrm{im}(\Bar{\varphi})$. Since $\prod_{t\in R} F_t$ is finite, $\mathrm{im}(\Bar{\varphi})$ is finite, yielding the conclusion.
\end{proof}

\subsection{Cohomological Characterization of the Narrow Mordell-Weil Group}
\label{subsection: cohomological characterization}
Using the same technique as \cite[Proposition 3.3]{FarbLoojenga_MordellWeil}, we give a cohomological characterization of the narrow Mordell-Weil group.
For $b\in B$ and a neighborhood $U\subset B$ containing no critical values in $U-\{b\}$,
let $\trans_0(M_b)$ be the group of diffeomorphisms of $M_b$ that extend to diffeomorphisms of $\pi^{-1}(U)$ that act by translation on the smooth fibers in $U$ and preserve all irreducible components of all fibers.
Denote by $\trans_0^c(M_b)$ the maximal compact subgroup of $\trans_0(M_b)$. Similarly to the discussion in \cite{FarbLoojenga_MordellWeil}, we express these local translation groups in terms of cohomology. In the smooth and nodal cases we recall their argument for completeness.

\begin{lemma}\label{lem: cohomological char for fiber}
    Let $\pi\colon M\to B$ be an elliptic fibration and $b\in B$. For any fiber $M_b$, there is a natural isomorphism
    \[
        \trans_0^c(M_b)\cong H^1(M_b;\R)/H^1(M_b;\Z).
    \]
\end{lemma}
\begin{proof}
    We distinguish cases based on whether $M_b$ is smooth, multiplicative or additive. 
    \begin{enumerate}[(1)]
        \item \textit{$M_b$ is smooth:} The universal cover of the torus $M_b$ is an affine space over $H_1(M_b;\R)$. The deck group of $M_b$ is naturally isomorphic to $H_1(M_b;\Z)$. Thus, there is a natural isomorphism
        \[
            \trans_0(M_b)\cong H_1(M_b;\R)/H_1(M_b;\Z)\cong H^1(M_b;\R)/H^1(M_b;\Z),
        \]
        where the second isomorphism is given by Poincaré duality. This group is compact, so it is also isomorphic to $\trans_0^c(M_b)$.

        \item \textit{$M_b$ is multiplicative:} 
        Let $C_0,\dots,C_r$ be the irreducible components of $M_b$, where $\sigma_0$ intersects $C_0$. The smooth locus $C_0^\#$ can be identified with $S^2-\{p_0,p_1\}$, for some $p_0\ne p_1$. Let 
        \[
        \trans(C_0^\#)\coloneqq \{f\in\diff(C_0^\#)\mid f \text{ extends to a diffeomorphism in } \trans_0(M_b)\}.
        \]
        The assumption that elements of $\trans_0(M_b)$ leave irreducible components invariant implies that any $f\in \trans_0(M_b)$ restricts to an element of $\trans(C_0^\#)$.
        Thus, there is a natural isomorphism 
        \begin{equation}\label{eq: trans_0 multipl 2}
            \trans_0(M_b)\cong \trans(C_0^\#).
        \end{equation}
        The universal cover of $C_0^\#$ is a torsor for $H_1(C_0^\#;\C)$.
        As in the previous case, the deck group of $C_0^\#$ is naturally isomorphic to $H_1(C_0^\#;\Z)$.
        Therefore,
        \begin{equation}\label{eq: trans_0 multiplicative}
            \trans(C_0^\#)\cong H_1(C_0^\#;\C)/H_1(C_0^\#;\Z).
        \end{equation}
        Let $R$ stand for one of the rings $\Z,\R,\C$.
        Alexander duality and the long exact sequence of the pair $(S^2,\{p_0,p_1\})$ imply 
        \begin{equation}\label{eq: H_1 C_0}
            H_1(C_0^\#;R)=H_1(S^2-\{p_0,p_1\};R)\cong \tilde H^0(\{p_0,p_1\};R) \cong H^1(S^2,\{p_0,p_1\};R).
        \end{equation}

        If $r=0$, let $q\in M_b$ be the node. Then $M_b-\{q\}\cong S^2-\{p_0,p_1\}$, so
        \[
            H^1(S^2,\{p_0,p_1\};R)\cong H^1(M_b,\{q\};R)\cong H^1(M_b;R).
        \]

        If $r>0$, excision and the long exact sequence of the pair $(M_b,C_1\cup\dots\cup C_{r})$, using that $C_1\cup\dots\cup C_{r}$ is simply-connected, give natural isomorphisms
        \[
            H^1(S^2,\{p_0,p_1\};R)\cong H^1(M_b,C_1\cup\dots\cup C_{r};R)\cong H^1(M_b;R).
        \]
        Hence, \eqref{eq: trans_0 multipl 2}, \eqref{eq: trans_0 multiplicative} and \eqref{eq: H_1 C_0} imply
        \[
            \trans_0(M_b)\cong H^1(M_b;\C)/H^1(M_b;\Z).
        \]

        The maximal compact subgroup of $H^1(M_b;\C)/H^1(M_b;\Z)$ is $H^1(M_b;\R)/H^1(M_b;\Z)$, yielding the conclusion.

        \item \textit{$M_b$ is additive:}  $\trans_0(M_b)\cong \C$ and its maximal compact subgroup is $\{0\}$. Since additive fibers are simply-connected, $H^1(M_b;R)$ is trivial and we obtain the desired identifications.   
    \end{enumerate}
\end{proof}

Let $\mathcal{T}(\pi)$ be the sheaf of smooth, fiberwise-translations on $B$: to every open set $U\subset B$ assign the group of diffeomorphisms of $\pi^{-1}(U)$, acting by translation on each smooth fiber. 
Denote by $\mathcal{T}_0(\pi)$ the subsheaf of $\mathcal{T}(\pi)$, that assigns to each open set $U\subset B$ the subgroup of $\mathcal{T}(\pi)(U)$ consisting of diffeomorphisms of $\pi^{-1}(U)$ that preserve each irreducible component of each singular fiber over $U$.
The group of global sections of $\mathcal{T}_0(\pi)$ is $\trans_0(\pi)$.
Let $\mathcal{E}_B$ be the sheaf of real-valued smooth functions on $B$.
Define the sheaf of translations belonging to a maximal compact subgroup by 
\[
    \mathcal{T}_0^c(\pi)\coloneqq (\mathcal{E}_B\otimes R^1\pi_*\Z)/R^1\pi_*\Z.
\]
We show that this can be viewed as a subsheaf of $\mathcal{T}_0(\pi)$ and determine how these two sheaves relate to each other.
\begin{lemma} \label{lem: T_0^c}
    There is a natural injective morphism of sheaves $\mathcal{T}_0^c(\pi) \xhookrightarrow{}\mathcal{T}_0(\pi)$. Furthermore, the cokernel of this morphism $\mathcal{T}_0(\pi)/\mathcal{T}_0^c(\pi)$ is supported on the critical values of $\pi$ and its non-zero stalks are vector spaces.
\end{lemma}
\begin{proof}
    For $U\subset B$ containing no critical values of $\pi$, Poincaré duality yields a natural isomorphism
    \[
        \iota\colon \mathcal{T}_0^c(\pi)(U)\to \mathcal{T}_0(\pi)(U).
    \]
    We show that $\iota$ can be extended to a morphism of sheaves $\mathcal{T}_0^c(\pi)\to \mathcal{T}_0(\pi)$.
    Let $b\in B$ be a critical value.
    \begin{enumerate}[(1)]
        \item If $M_b$ is multiplicative, there exists a small neighborhood $U$ such that the explicit degeneration of $M_b$ gives an action 
        \[
            \rho\colon \pi^{-1}(U)\times (H^1(M_b;\R)/H^1(M_b;\Z))\to \pi^{-1}(U)
        \]
        by component-preserving fiber translations.
        Hence, for any smooth function $a\colon U \to H^1(M_b;\R)/H^1(M_b;\Z)$, there is a translation $F_a$ of $\pi^{-1}(U)$, defined by $F_a(x)=\rho(x,a(\pi(x)))$.

        \item If $M_b$ is additive, let $s\in \mathcal{T}_0^c(U)$. The stalk $\mathcal{T}_0^c(\pi)_b$ is $\{0\}$, since $M_b$ is simply-connected, so $H^1(M_b;\R)\cong 0$. 
        Thus, there exists $V\subseteq U$ such that $s|_V=0$ and we can define $\iota(s|_V)=0$. This agrees with the construction on the smooth locus.
    \end{enumerate}
    These constructions agree on the smooth locus, so they yield a morphism of sheaves $\iota\colon\mathcal{T}_0^c(\pi)\to \mathcal{T}_0(\pi)$. Injectivity of $\iota$ follows immediately from injectivity at the level of stalks.

    In the smooth case, the inclusion $\iota\colon \mathcal{T}_0^c(\pi)\xhookrightarrow{} \mathcal{T}_0(\pi)$ is an isomorphism, so the stalk at any regular value of $\pi$ is trivial.
    Let $C_{B,b}^\infty (R)$ denote the group of germs at $b$ of smooth functions $B\to R$.
    At a critical value $b\in B$ the discussion above and \Cref{lem: cohomological char for fiber} imply
    \[
        (\mathcal{T}_0(\pi)/\mathcal{T}_0^c(\pi))_b=
        \begin{cases}
            C^\infty_{B,b}(\R) \text{, if } b \text{ is multiplicative,} \\
            C^\infty_{B,b}(\C) \text{, if } b \text{ is additive.}
        \end{cases}
    \]
    
\end{proof}

\begin{proposition}[Cohomological characterization] \label{prop: MW0=H1} Consider an elliptic fibration $\pi\colon M\to B$ with finitely many critical values\footnote{This is automatically satisfied if $B$ is a projective curve.}.
There exists a natural isomorphism
    \[
        \delta\colon \MW_0(\pi)\xrightarrow[]{\cong} H^1(B,R^1\pi_*\Z)
    \]
induced by the boundary map in the long exact sequence in cohomology associated to the short exact sequence of sheaves
    \begin{equation} \label{eq: SES sheaves in coh char}
        0\to R^1\pi_*\Z \to \mathcal{E}_B\otimes_\Z R^1\pi_*\Z \to \mathcal{T}_0^c(\pi)\to 0.
    \end{equation}
\end{proposition}
\begin{proof}
    The short exact sequence of sheaves \eqref{eq: SES sheaves in coh char} follows from the definition of $\mathcal{T}_0^c(\pi)$.
    Since $\mathcal{E}_B\otimes_\Z R^1\pi_*\Z$ is a fine sheaf of $\mathcal{E}_B$-modules, $H^q(B,\mathcal{E}_B\otimes_\Z R^1\pi_*\Z)=0$ for any $q>0$.
    Thus, the long exact sequence in cohomology associated to \eqref{eq: SES sheaves in coh char} gives the following exact sequence of topological groups
    \[
        0\to H^0(B,R^1\pi_*\Z)\to H^0(B,\mathcal{E}_B\otimes_\Z R^1\pi_*\Z) \to H^0(B,\mathcal{T}_0^c(\pi)) \xrightarrow[]{\delta}H^1(B,R^1\pi_*\Z)\to 0.
    \]
    The space $H^0(B,\mathcal{E}_B\otimes_\Z R^1\pi_*\Z)$ is connected, since it is a vector space and $H^1(B,R^1\pi_*\Z)$ is discrete, so exactness gives an identification 
    \[
        \pi_0(H^0(B,\mathcal{T}_0^c(\pi))) \cong H^1(B,R^1\pi_*\Z).
    \]
    
    By \Cref{lem: T_0^c}, the quotient $\mathcal{T}_0(\pi)/\mathcal{T}_0^c(\pi)$ has support contained in the singular values of $\pi$ and the stalks are contractible. 
    Hence, $H^0(B,\mathcal{T}_0/\mathcal{T}_0^c)$ is contractible. Using disjoint neighborhoods of the critical values and smooth cut-off functions, 
    there is a continuous section of    
    \[
        H^0(B,\mathcal{T}_0)\to H^0(B,\mathcal{T}_0/\mathcal{T}_0^c).
    \]
    Thus, there is a homeomorphism
    \[
        H^0(B,\mathcal{T}_0)\cong H^0(B,\mathcal{T}_0^c)\times H^0(B,\mathcal{T}_0/\mathcal{T}_0^c).
    \]
    Hence, the inclusion $\mathcal{T}_0^c\xhookrightarrow{} \mathcal{T}_0$ induces an isomorphism on the path component groups of $H^0(B,\mathcal{T}_0^c(\pi))$ and $H^0(B,\mathcal{T}_0(\pi))$.
    By definition, $\MW_0(\pi)\cong \pi_0(H^0(B,\mathcal{T}_0(\pi)))$, implying the conclusion.
\end{proof}

\section{Computation of the Narrow Mordell-Weil Group}\label{section: narrow MW}
In this section, we compute the Narrow Mordell-Weil group of an arbitrary elliptic fibration over a projective base. We describe $\MW_0(\pi)$ as a subquotient of $H^2(M)$ and prove \Cref{thm: MW0 as perp of triv}. Finally, we compute the rank of $\MW_0(\pi)$.

\subsection{A Description of the Narrow Mordell-Weil Group Using Annihilators}
\label{subsection: MW0 as annihilator}
We start with a lemma about Eichler transformations.
\begin{lemma}[Eichler transformations]\label{lem: eichler transformation}
    Let $L$ be a unimodular lattice, $e\in L$ a primitive isotropic vector and $a_1,\dots,a_n\in e^\perp$, such that the Gram matrix $(a_i\cdot a_j)_{i,j}$ is nondegenerate. Let
    \[
        K\coloneqq e^\perp \cap a_1^\perp\cap\dots\cap a_n^\perp.
    \]
    Let $\tau$ be an isometry of $L$ such that $\tau(e)=e$ and $\tau(a_i)=a_i, \forall i=1,\dots,n$.
    If $\tau$ acts trivially on the successive quotients of the filtration $0\subset \Z e\subset K\subset L$,
    then there exists a unique class $\bar v\in K/\Z e$ such that $\tau$ acts on $L$ as the Eichler transformation $E_{e,v}$ given by
    \[
        E_{e,v}(x)=x+(x\cdot e)v-(x\cdot v)e-\frac{1}{2}(v\cdot v)(x\cdot e)e,
    \]
    where $v\in K$ is a lift of $\bar v$ with $v\cdot v$ even.
\end{lemma}
\begin{proof}
    Define a homomorphism $w\colon e^\perp\to K$, $w(x)=\tau(x)-x$. The map $w$ is well-defined, since $\tau$ acts trivially on $L/K$.
    Since $\tau$ is an isometry, $\tau(x)\cdot \tau (y)=x\cdot y$ for all $x,y\in e^\perp$. Substituting $\tau(x)=x+w(x)$ implies
    \begin{equation}\label{eq: w acting on x,y in e perp}
        x\cdot w(y)+w(x)\cdot y+w(x)\cdot w(y)=0.
    \end{equation}
    For any $y\in K\subset e^\perp$, the assumption that $\tau$ acts trivially on $K/\Z e$ implies that $w(y)\in \Z e$. Since ${x,w(x)\in e^\perp}$, \eqref{eq: w acting on x,y in e perp} implies that for any $x\in e^\perp$ and $y\in K$, $w(x)\cdot y=0$. Thus, $w(x)\in K^\perp$. Therefore, $w(x)\in K\cap K^\perp$. 
    
    The fact that $(a_i\cdot a_j)_{i,j}$ is nondegenerate implies that $K\cap K^\perp=\Z e$. Indeed, this implies that $a_1,\dots,a_n$ are linearly independent over $\Q$, so denoting their $\Q$-span by $V$ implies $V\cap V^\perp=0$. Let $K_\Q\coloneqq K\otimes_{\Z} \Q$. Then $K_{\Q}=(\Q e\oplus V)^\perp$ and $K_{\Q}^\perp=\Q e\oplus V$.
    Thus, $K_\Q\cap K_\Q^\perp= \Q e$. Since $e$ is primitive, $L\cap \Q e=\Z e$, so $K\cap K^\perp = \Z e$.
    Hence, we can restrict the codomain to $w\colon e^\perp \to \Z e$. Thus, there exists a homomorphism $\varphi\colon e^\perp\to \Z$ such that $w(x)=\varphi(x)e$.

    Since $e$ is a primitive vector in the unimodular lattice $L$, there exists $f\in L$ such that $e\cdot f=1$. The homomorphism $\varphi$ can be extended to $L=e^\perp \oplus \Z f$ by setting $\varphi(f)=0$. Since $L$ is unimodular, there exists a unique $v\in L$ such that $\varphi(x)=-x\cdot v,\forall x\in L$. Now, $\tau(e)=e$ implies $w(e)=0$, so $\varphi(e)=0$, which means $v\in e^\perp$. Similarly, $\tau(a_i)=a_i$ implies $v\in a_i^\perp$, $\forall i$ so $v\in K$.

    We need to determine the action of $\tau$ on $f$. By assumption, there exists $v'\in K$, such that $\tau(f)=f+v'$. For $x\in e^\perp$, the fact that $\tau$ is an isometry implies
    \[
        x\cdot f = \tau(x)\cdot \tau(f)=(x-(x\cdot v)e)\cdot (f+v').
    \]
    Expanding and using $e\cdot f=1$ and $e\cdot v'=0$ implies $x\cdot v'=x\cdot v$. Since $x\in e^\perp$ was arbitrary, it follows $v'-v\in (e^\perp)^\perp$ and since $e$ is primitive, $v'-v=\lambda e$ for $\lambda\in \Z$. Thus, $\tau(f) = f+v+\lambda e$. Using this in the identity $f\cdot f=\tau(f)\cdot \tau(f)$ implies $\lambda=-f\cdot v-\frac{1}{2} v\cdot v$. Note that $\lambda\in \Z$ implies that $v\cdot v$ is even. Hence, if $x\in e^\perp$ or $x\in \Z f$, $\tau(x)=x+(x\cdot e)v-(x\cdot v)e-\frac{1}{2}(v\cdot v)(x\cdot e)e$. Since $e^\perp$ and $f$ generate $L$, this formula holds for $x\in L$.

    If $E_{e,v}=E_{e,v'}$ for $v,v'\in K$, taking $x\in e^\perp$ implies $x\cdot v=x\cdot v'$, so $v-v'\in (e^\perp)^\perp$. Hence, primitivity of $e$ implies that $v-v'\in \Z e$. Thus, there is a unique class $\Bar{v}\in K/\Z e$ such that $\tau$ acts on $L$ by a lift of $\Bar{v}$.
\end{proof}

For a homology class $v\in H_2(M)$, let $H^2(M)^v$ denote the annihilator of $v$ under the evaluation pairing $H^2(M)\to \Z, \xi \mapsto \langle \xi, v \rangle$. By $v^\vee\in H^2(M)$, we denote the Poincaré dual of $v$ in $M$.
We use the convention that $f\in\diff(M)$ acts on $H^2(M)$ on the left by $(f^{-1})^*$.
The following proposition expresses $\MW_0(\pi)$ in terms of annihilators in $H^2(M)$. In the proof, we include a computation of the edge maps of the Leray spectral sequence of $\pi$, which we did not find in the literature.

\begin{proposition} \label{prop: narrow MW as annihilator}
Let $\pi\colon M\to B$ be an elliptic fibration with $B$ a projective curve of genus $g$. 
Denote by $[F]$ the fiber class of $\pi$ and let $[F]^\vee$ be its Poincaré dual. Let us label the critical values in $B$ by $t_1,\dots, t_n$ and let the irreducible components of the singular fiber $\pi^{-1}(t_i)$ be denoted $C_{i,0},\dots, C_{i,r_i}$, where $C_{i,0}$ is the component intersecting $\sigma_0$.
Then there is a natural isomorphism
\[
\begin{array}{c c l}
     \MW_0(\pi)&\xrightarrow[]{\cong} &
        \left(H^2(M)^{[F]}\cap \bigcap\limits_{i=1}^n \bigcap\limits_{j=1}^{r_i} H^2(M)^{[C_{i,j}]}\right) / \Z[F]^\vee, \\
     \tau&\mapsto & \Bar{c}_\tau=[\tau([\sigma_0]^\vee)-[\sigma_0]^\vee].
\end{array}
\]
Furthermore, $\tau$ acts on $H^2(M)$ by the Eichler transformation
\[
    E_{[F]^\vee,c_\tau}(x)=x+(x\cdot [F]^\vee)c_\tau - (x\cdot c_\tau) [F]^\vee - \frac{1}{2} (c_\tau\cdot c_\tau)(x\cdot [F]^\vee)[F]^\vee,
\]
where $c_\tau$ is a lift of $\Bar{c}_\tau$.
\end{proposition}
\begin{proof}
    Consider the Leray spectral sequence 
    \[
        E_2^{p,q}\coloneqq H^p(B,R^q\pi_*\Z) \implies H^{p+q}(M).
    \]
    The only possible nonzero elements on the $E_2$ page are the ones below:
    
    \begin{center}
        \begin{tikzcd}[row sep=huge, column sep=huge]
        H^0(B, R^2\pi_*\mathbb{Z})\arrow[rrd, red, "A_2"] & H^1(B, R^2\pi_*\mathbb{Z}) & H^2(B, R^2\pi_*\mathbb{Z})\\
        H^0(B, R^1\pi_*\mathbb{Z})\arrow[rrd, red, "A_1"] & H^1(B, R^1\pi_*\mathbb{Z}) & H^2(B, R^1\pi_*\mathbb{Z})\\
        H^0(B, R^0\pi_*\mathbb{Z}) & H^1(B, R^0\pi_*\mathbb{Z}) & H^2(B, R^0\pi_*\mathbb{Z})
    \end{tikzcd}
    \end{center}

    Thus, the spectral sequence must degenerate at the latest on the $E_3$ page. The only possibly nonzero differentials on $E_2$ are
        \[
        A_1\colon H^0(B,R^1\pi_*\Z)\to H^2(B,R^0\pi_*\Z)
        \]
        and 
        \[
        A_2\colon H^0(B,R^2\pi_*\Z)\to H^2(B,R^1\pi_*\Z).        
        \]

    \step\textit{Existence of a section implies that $A_1\equiv0$.} 
    We first show that the edge map 
    \[
        \mathrm{edge}_1\colon H^2(B,R^0\pi_*\Z)\cong H^2(B)\to H^2(M)
    \]
    is the pullback of $\pi\colon M\to B$. Consider the trivial fibration $\id\colon B\to B$. The morphism of fibrations
    \begin{center}
        \begin{tikzcd}
        M \arrow[r,"\pi"]\arrow[d,"\pi"] & B \arrow[d,"\id_B"]\\
        B \arrow[r,"\id_B"] & B
    \end{tikzcd}
    \end{center}
    induces a morphism of Leray spectral sequences. By naturality of the edge map, there is a commutative diagram
    \begin{center}
        \begin{tikzcd}
        H^2(B,R^0\id_* \Z) \arrow[r,"\id"]\arrow[d,"\id"] & H^2(B,R^0\pi_* \Z) \arrow[d,"\mathrm{edge}_1"]\\
        H^2(B) \arrow[r,"\pi^*"] & H^2(M).
    \end{tikzcd}
    \end{center}

    Clearly, $R^0\id_* \Z$ is equal to the constant sheaf $\underline\Z$. The left vertical map is the edge map for the trivial fibration $\id\colon B\to B$, so it is the identity.
    Furthermore, all fibers of $\pi$ are connected (see Kodaira's classification), so $R^0\pi_* \Z=\underline\Z$. 
    The top horizontal map is induced by the isomorphism of sheaves $R^0\id_*\Z\to R^0\pi_*\Z$ and is therefore an isomorphism. 
    By naturality, the bottom map $H^2(B)\to H^2(M)$ is the pullback $\pi^*$, so $\mathrm{edge}_1=\pi^*$.
    
    The existence of a section $\sigma_0$ implies that $\pi\circ \sigma_0=\id_B$. Thus, the induced map $\pi^*\colon H^2(B)\to H^2(M)$ is injective. By definition of the edge map, $\ker(\mathrm{edge_1})=\mathrm{im}(A_1)$. Hence, $A_1=0$.

    \step \textit{The edge map $\mathrm{edge}_2\colon  H^2(M)\to H^0(B,R^2\pi_*\Z)$ is given by restriction:} \label{step: computation edge 2}
    \[
        \alpha \mapsto (\alpha|_{M_b})_{b\in B}.
    \]
    The fibration $f\colon M\to pt$ fits into the following commutative diagram:
   
    \begin{center}
        \begin{tikzcd}
        M \arrow[r,"\id_M"]\arrow[d,"\pi"] & M \arrow[d,"f"]\\
        B \arrow[r,"g"] & pt.
    \end{tikzcd}
    \end{center}
    Naturality of the Leray spectral sequence and of the edge map implies that there is a commutative diagram
    \begin{center}
        \begin{tikzcd}
        H^2(M) \arrow[r,"\id"]\arrow[d,"\mathrm{edge}_f"] & H^2(M) \arrow[d,"\mathrm{edge}_2"]\\
        H^0(pt,R^2f_*\Z) \arrow[r,"g^*"] & H^0(B,R^2\pi_*\Z).
    \end{tikzcd}
    \end{center}
    Since $R^2f_*\Z = H^2(M)$, there is a natural isomorphism $H^0(pt,R^2f_*\Z)\cong H^2(M)$. 
    The top horizontal map is the identity, since it is the pullback of the identity.
    The left vertical map is the edge map for the Leray filtration of $f\colon M\to pt$. The $E_2$ page of this Leray filtration only has nonzero elements for $p=0$, so the edge map $H^2(M)\to H^0(pt,R^2 f_*\Z)\cong H^2(M)$ is the identity. The bottom horizontal map corresponds to the map on global sections given by the sheaf morphism $g^{-1}R^2f_*\Z\to R^2\pi_*\Z$, induced by $\id_M$. On the stalk above $b\in B$, this map is given by restriction $H^2(M)\to H^2(M_b)$. Thus, the map on  global sections is also induced by restriction.
    Therefore, $\mathrm{edge}_2$ is given by restriction.
    
    \step \textit{Interpret Leray Filtration.} \label{step: interpret Leray}
    Let the Leray filtration of $H^2(M)$ be given by
    \[
        H^2(M)=L_2H^2\supset L_1H^2 \supset L_0H^2=\pi^* H^2(B) \supset 0.
    \]

    Then $L_1H^2 / L_0H^2 \cong H^1(B,R^1\pi_*\Z)$ and $L_2H^2 / L_1H^2 \cong \ker A_2=\mathrm{im}(\mathrm{edge}_2)$. By definition, there exists the following short exact sequence:

    \begin{equation}\label{eq: SES Leray SS}
        0 \xrightarrow[]{}L_1H^2 \xrightarrow[]{} H^2(M) \xrightarrow[]{\mathrm{edge}_2} \mathrm{im}(\mathrm{edge}_2)  \xrightarrow[]{}0.
    \end{equation}

    By Step \ref{step: computation edge 2}, $\mathrm{edge}_2\colon H^2(M)\to H^0(B,R^2\pi_*\Z)$ is given by restricting to each fiber.
    Hence, 
    \[
        L_1H^2=\ker(\mathrm{edge}_2)=\{\alpha\in H^2(M)\mid \alpha|_{[M_b]}=0, \forall b\in B\}.
    \]
    For a singular fiber $M_b$,
    $H_2(M_b)$ is freely generated by its irreducible components. Hence, $\alpha|_{M_b}$ is trivial if and only if it evaluates trivially on each irreducible component. Thus,
    \[
        L_1H^2=H^2(M)^{[F]}\cap \bigcap_{i,j} H^2(M)^{[C_{i,j}]}.  
    \]
    For each $i$, $\alpha|_{[F]}=0$ and $\alpha|_{[C_{i,j}]}=0$ for $j>0$ implies $\alpha|_{[C_{i,0}]}=0$, so the latter does not need to be explicitly included in the intersection.
    
    Furthermore, $L_0H^2=\pi^* H^2(B)$ is spanned by the Poincaré dual $[F]^\vee$ of the fiber class.
    Therefore by \Cref{prop: MW0=H1}, 
    \begin{equation*}\label{eq: formula for MW_0}
        \MW_0(\pi)\cong H^1(B, R^1\pi_*\Z) \cong L_1H^2 / L_0H^2 \cong 
        \left(H^2(M)^{[F]}\cap \bigcap_{i=1}^n \bigcap_{j=1}^{r_i} H^2(M)^{[C_{i,j}]}\right)  / \Z[F]^\vee.
    \end{equation*}

\step \textit{$\MW_0(\pi)$ acts by Eichler transformations.}
We first show that $\MW_0(\pi)$ acts trivially on the grading of the Leray filtration. For $q\ge 0$, the stalk of the sheaf $R^q\pi_*\Z$ at $b\in B$ is $H^q(M_b)$. 
Since $\MW_0(\pi)$ acts on $M$ by fiberwise translation, it acts trivially on $H^0(M_b)$.
If $M_b$ is smooth, a translation acts trivially on $H^1(M_b)$. If $M_b$ is multiplicative, an element of $\MW_0(\pi)$ acts isotopically trivial on $M_b$,
so it fixes the orientation class and therefore it acts trivially on $H^1(M_b)$. If $M_b$ is additive, $H^1(M_b)=0$, so the action is clearly trivial.
By definition, the action of $\MW_0(\pi)$ fixes the irreducible components of all fibers, so it acts trivially on $H^2(M_b)$. For $q>2$, $H^q(M_b)=0$.
The trivial action on the stalks implies that $\MW_0(\pi)$ acts trivially on each sheaf $R^q\pi_*\Z$. Therefore, $\MW_0(\pi)$ acts trivially on $E_2^{p,q}=H^p(B,R^q\pi_*\Z)$. It follows that it acts trivially on $E_{\infty}^{p,q}$, which by definition is the grading of $H^2(M)$ with respect to the Leray filtration.

We apply Lemma \ref{lem: eichler transformation} to the Leray filtration on $H^2(M)$, where $e=[F]^\vee$, $a_{i,j}=[C_{i,j}]^\vee$ and $\tau\in \MW_0(\pi)$. The existence of a section $\sigma_0$ implies that $e$ is primitive, since $[F]^\vee\cdot [\sigma_0]^\vee=1$. Furthermore, all fibers of $\pi$ represent the same homology class, so $e$ is isotropic.
The Gram matrix associated to $[C_{i,j}]^\vee$ is block diagonal with blocks the negative Cartan matrices associated to the singular fibers, so each block is nondegenerate, implying that the entire matrix is nondegenerate. By definition of $\MW_0(\pi)$, $\tau([C_{i,j}]^\vee)=[C_{i,j}]^\vee$. Hence, there exists a unique $\bar c_\tau\in L_1H^2/\Z[F]^\vee$, such that $\tau$ acts on $H^2(M)$ by the Eichler transformation $E_{[F]^\vee,c_\tau}$. The isomorphism in Step \ref{step: interpret Leray} is given as the composition
\[
    \MW_0(\pi)\xrightarrow[]{\delta} H^1(B,R^1\pi_*\Z) \xrightarrow[]{\cong} L_1H^2/L_0H^2.
\]
By writing out the definition of $\delta$ as the boundary map and of the isomorphism from the $E^{1,1}_\infty$ term of the Leray spectral sequence to the corresponding graded piece of the Leray filtration,
one can check that this composition takes an element $\tau\in \MW_0(\pi)$ to $[\PD(\tau[\sigma])-\PD([\sigma])]\in L_1H^2/L_0H^2$.
Since for any section $\sigma\colon B\to M$,
$\tau([\sigma]^\vee)\equiv [\sigma]^\vee + c_\tau$ modulo $\Z[F]^\vee$, the map $\tau\mapsto c_\tau$ is the group isomorphism from Step \ref{step: interpret Leray}. 
\end{proof}

\begin{remark}
    In the case of a generic elliptic fibration $\pi\colon M\to B$, Farb and Looijenga proved in \cite{FarbLoojenga_MordellWeil} that a section of $\pi$ always exists by showing that $H^2(B,R^1\pi_*\Z)=0$. This is not true if arbitrary fibers are allowed. In fact, the assumption that a section exists is crucial for our proof.
\end{remark}

\begin{remark}
    Over $\Q$ the spectral sequence in the proof of \Cref{prop: narrow MW as annihilator} degenerates on the $E_2$ page, since the edge map $\mathrm{edge}_2$ in the proof of \Cref{prop: narrow MW as annihilator} has full rank.
    Indeed, restricting the domain to the lattice generated by $[F]^\vee,[\sigma_0]^\vee$ and $[C_{i,j}]^\vee$ for $j>0$, $\mathrm{edge}_2$ can be computed explicitly as a matrix.
    More generally, Deligne showed that for a family of smooth projective varieties, the Leray-Serre spectral sequence over $\Q$ degenerates on $E_2$.
    However, working over $\Z$, one needs to consider the $E_3$ page.
\end{remark}

\subsection{The Narrow Mordell-Weil Group in Terms of the Trivial Lattice}
\label{subsection: proof main thm}

We use \Cref{prop: narrow MW as annihilator} to prove our main theorem.
\begin{proof}[Proof of \Cref{thm: MW0 as perp of triv}]
    By Poincaré duality, $\alpha\in H^2(M)$ lies in the annihilator of $a\in H_2(M)$ if and only if $\alpha\smile \PD(a)=0$. Let 
    \[
        L\coloneqq \Z[F]^\vee\oplus \bigoplus_{i=1}^n\bigoplus_{j=1}^{r_i}\Z[C_{i,j}]^\vee.
    \]
    For each $1\le i \le n$, 
    \[
        [C_{i,0}]^\vee=[F]^\vee-\sum_{j=1}^{r_i} m_{i,j}[C_{i,j}]^\vee,
    \]
    so $[C_{i,0}]^\vee\in L$. Thus, $\triv(\pi)=L\oplus \Z[\sigma_0]^\vee$. The sum is direct since $\forall x\in L, x\cdot [F]^\vee=0$, but $[\sigma_0]^\vee\cdot [F]^\vee=1$.
    \Cref{prop: narrow MW as annihilator} implies
    \[
        \MW_0(\pi)\cong L^\perp /\Z [F]^\vee,
    \]
    where $\perp$ is taken in $H^2(M;\Z)$ with respect to the intersection form.
    Consider the projection
    \[
    \begin{array}{ccc}
         \Phi\colon L^\perp &\to& \triv(\pi)^\perp,\\
         x&\mapsto &x- (x \cdot [\sigma_0]^\vee) [F]^\vee.  
    \end{array}
    \]
    $\Phi$ is well-defined, since $[F]^\vee \cdot [\sigma_0]^\vee=1$, so $\Phi(x)\cdot [\sigma_0]^\vee=0$ and $x\in L^\perp$ implies $\Phi(x)\in L^\perp$.
    It preserves the intersection form, since elements in the domain satisfy $x\cdot [F]^\vee=0$ and $[F]^\vee\cdot [F]^\vee=0$.
    Furthermore, $\Phi$ is surjective since $\triv(\pi)^\perp \subset L^\perp$ and $\Phi|_{\triv(\pi)^\perp}=\id$. Finally, $\Phi$ has kernel $\Z [F]^\vee$.
    Therefore, $\Phi$ induces the desired isomorphism of lattices.
\end{proof}

\begin{remark} \label{rmk: unimodularity}
    The group $\MW_0(\pi)$ is not unimodular in general. 
    Let $T\coloneqq \triv(\pi)$ and denote by $\Bar{T}\coloneqq (T\otimes \Q)\cap H^2(M)$ its primitive closure.
    By Poincaré duality, $H^2(M)$ is unimodular, so 
    \[
        |\disc(\MW_0(\pi))|=|\disc(T^\perp)|=|\disc(\Bar{T})|=\frac{|\disc(T)|}{[\Bar{T}:T]^2}.
    \]
    
    Using the notation from Proposition \ref{prop: narrow MW as annihilator}, $\triv(\pi)$ is isomorphic to the orthogonal direct sum of the lattices $\Z[F]^\vee\oplus\Z[\sigma_0]^\vee$ and $\bigoplus_{j>0} \Z[C_{i,j}]^\vee$. The Gram matrix of $\Z[F]^\vee\oplus\Z[\sigma_0]^\vee$ is $\begin{pmatrix}
        0 & 1\\
        1 & [\sigma_0]^2
    \end{pmatrix}$, so this lattice has determinant $-1$.
    By the discussion in \Cref{section: Kodaira classification}, the Gram matrix of $\bigoplus_{j>0} \Z[C_{i,j}]^\vee$ is the negative Cartan matrix $A_{t_i}$ of the Dynkin diagram associated to the singular fiber $M_{t_i}$. 

    Since $T\otimes \Q\subset H^{1,1}(M)$, the primitive closure of $T$ in $H^2(M)$ is equal to the primitive closure of $T$ in $\NS(M)$. Hence,
    \cite[Corollary 6.29]{Schutt_Schioda_MordellWeil} implies $[\Bar{T}:T]= |\MWH(\pi)_{\text{tors}}|$. Thus,
    \[
        |\disc(\MW_0(\pi))|=\frac{\prod |\det(A_{t_i})|}{|\MWH(\pi)_{\text{tors}}|^2}.
    \]
    Hence, whether $\MW_0(\pi)$ is unimodular depends on the types of singular fibers and the size of the torsion subgroup of $\MWH(\pi)$.
\end{remark}

\subsection{The Rank of the Narrow Mordell-Weil Group}
\label{subsection: rank computation of MW0}
We use \Cref{thm: MW0 as perp of triv} to compute the rank of $\MW_0(\pi)$. Recall that $\pi\colon M\to \Sigma_g$ is an elliptic fibration and $\Sigma_g$ is a projective curve of genus $g$. Hence, $M$ is a projective complex surface.
We assume that $\pi$ has $a\ge 0$ additive fibers and $m\ge 0$ multiplicative fibers.
\begin{proof}[Proof of \Cref{cor: rk MW_0 thm}]
    By \Cref{thm: MW0 as perp of triv},
    \begin{equation}\label{eq: rk MW0 H^2 - triv}
        \rk(\MW_0(\pi))=\rk(H^2(M))-\rk(\triv(\pi)).
    \end{equation}

    By \Cref{prop: fund group}, $\pi_1(M)\cong\pi_1(\Sigma_g)$. This implies $H_1(M)\cong \Z^{2g}$. By the universal coefficient theorem for cohomology, $H^1(M)\cong \Z^{2g}$ and by Poincaré duality, $H^3(M)\cong \Z^{2g}$. Thus, 
    \begin{equation*} 
    \begin{aligned}
        \rk(H^2(M))&=\chi(M)+\rk(H^1(M))+\rk(H^3(M))-\rk(H^0(M))-\rk(H^4(M)) =\\
        &= \chi(M)+4g-2= \sum_{i=1}^n \chi (M_{t_i})+4g-2,
    \end{aligned}
    \end{equation*}
    where the last equality follows from \Cref{lem: euler char is sum of euler of fibers}.
    By \Cref{lem: euler char of sing fibers}, $\chi(M_{t_i})=r_i+1$ if $M_{t_i}$ is multiplicative and 
    $\chi(M_{t_i})=r_i+2$ if $M_{t_i}$ is additive. This, together with $n=a+m$, implies

    \begin{equation*}
        \rk(H^2(M))=\sum_{i=1}^n (r_i+1)+a+4g-2=
        \sum_{i=1}^n r_i+2a+m+4g-2.
    \end{equation*}
    The generators $[\sigma_0]^\vee$, $[F]^\vee$ and $[C_{i,j}]^\vee$ for $j>0$ of $\triv(\pi)$ are linearly independent by the Gram matrix argument in \Cref{rmk: unimodularity}.
    Thus,
    \[
        \rk(\triv(\pi))= \sum_{i=1}^n r_i+2.
    \]
    Therefore, \eqref{eq: rk MW0 H^2 - triv} implies $\rk(\MW_0(\pi))=2a+m+4g-4$.
    Since $\triv(\pi)^\perp$ is a subgroup of the free abelian group $H^2(M)$, it is free abelian, so $\MW_0(\pi)$ is free abelian.
\end{proof}

\section{Applications}
We apply our theorems to compare the smooth and holomorphic Mordell-Weil groups. We then compute the narrow Mordell-Weil group in some concrete examples.

\subsection{Comparing Smooth and Holomorphic Versions}
\label{subsection: comparing smooth and hol}
We use the formula for the rank of $\MW(\pi)$ to relate it to the rank of $\MWH(\pi)$.
\begin{proof}[Proof of Theorem \ref{thm: intro smooth vs hol}]
    We use the Shioda-Tate formula to find an upper bound for the rank of $\MWH(\pi)$ and relate it to the rank of $\MW(\pi)$ determined in \Cref{cor: rk MW}.

    Recall that $R\subset \Sigma_g$ is the set of critical values of $\pi$.
    For each $b\in R$, denote by $\mathrm{nr}(b)$ the number of irreducible components in $\pi^{-1}(b)$. 
    Let $d$ denote the arithmetic genus of $M$ and let $h^{i,j}\coloneqq \dim_\C H^{i,j}(M)$ be the Hodge numbers of $M$.
    Recall that the Néron-Severi lattice $\NS(M)$ is a subset of $H^{1,1}(M)$. Hence, ${\rk(\NS(M))\le h^{1,1}}$.

    We compute the Hodge number $h^{1,1}$.
    By \Cref{prop: fund group}, $\pi_1(M)\cong \pi_1(\Sigma_g)$, the Hurewicz theorem, and the universal coefficient theorem, $H^1(M;\C)\cong \C^{2g}$. The Hodge theorem implies
    \[
        h^{1,0}+h^{0,1}=\dim_\C(H^1(M;\C))=2g.
    \]
    Complex conjugation provides an isomorphism $H^{1,0}(M)\to H^{0,1}(M)$, so $h^{1,0}=h^{0,1}=g$. 
    Since $h^{0,0}=1$ and $d=\chi_\text{hol}(M)=h^{0,0}-h^{1,0}+h^{2,0}$.
    It follows that $h^{2,0}=d+g-1$. Since $h^{2,0}=h^{0,2}$, the Hodge theorem implies that 
    \[
        2(d+g-1)+h^{1,1}=h^{2,0}+h^{0,2}+h^{1,1}=\dim_\C(H^2(M;\C)).
    \]
    Furthermore,
    \[
    \begin{aligned}
        12d=\chi(M)=\sum_{i=0}^4 (-1)^i \dim_\C(H^i(M;\C)) 
        = 1-2g+\dim_\C(H^2(M;\C))-2g+1.
    \end{aligned}
    \]
    Hence, $h^{1,1}=10d+2g$, which implies 
    \begin{equation}\label{eq: bound on NS}
        \rk(\mathrm{NS}(M))\le 10d+2g.
    \end{equation}

    The Euler characteristic of $M$ is $12d$. Since $M$ is a fibration over $B$ and the Euler characteristic of a torus is zero, the Euler characteristic of $M$ is also the sum of the Euler characteristics of the singular fibers.
    \Cref{lem: euler char is sum of euler of fibers} and \Cref{lem: euler char of sing fibers} imply
    \[
        12d=\chi(M)=\sum_{x\in R}\chi(\pi^{-1}(x))=
        \sum_{x\in R} \mathrm{nr}(x)+a.
    \]
    The Shioda-Tate formula (\Cref{cor: Shioda Tate Formula}) and \eqref{eq: bound on NS} imply
    \[
    \rk(\MW_\text{hol}(\pi))\le 10d+2g-2-\sum_{x\in R} \mathrm{nr}(x) + |R| = 2g+2a+m-2-2d.
    \]
    By \Cref{cor: rk MW}, $\rk(\MW(\pi))=4g+2a+m-4$, which yields
    \[
        \rk(\MW(\pi))-\rk(\MW_\text{hol}(\pi))\ge 2g+2d-2.
    \]    
\end{proof}

\begin{proof}[Proof of \Cref{cor: rank MW strictly bigger than MWH unless rational}]
    Since $d\ge1$ and $g\ge 0$, the inequality is immediate from  \Cref{thm: intro smooth vs hol}. Equality is attained if and only if $g=0$ and $d=1$, which implies that $M$ is a rational elliptic surface. Conversely, if $\pi\colon M\to B$ is a rational elliptic surface, $g=0$ and $\rk(\MW(\pi))=2a+m-4$. Rationality of $M$ implies $h^{2,0}=0$, so $\rk(\NS(M))=\dim_\C H^{1,1}(M)=\rk (H^2(M))=10$. By \Cref{cor: Shioda Tate Formula},
    \[
    \begin{aligned}
        \rk(\MWH(\pi))&=\rk(\NS(M))-2-\sum_{b\in R}(\mathrm{nr}(b)-1) =\\
        &=8-\sum_{b\in R} \chi(M_b)+2a+m=8-12+2a+m=\rk(\MW(\pi)).
    \end{aligned}
    \]
\end{proof}

\subsection{Examples}
\label{subsection: examples}
Our theorems imply that the rank of the (narrow) Mordell-Weil group of an elliptic surface is easy to compute once the types of singular fibers of an elliptic surface are known. 
We use this to give examples of different elliptic surfaces with isomorphic narrow Mordell-Weil groups.

\begin{ex}
    The elliptic fibration $\pi\colon M\to \p^1$ obtained as the minimal compactification of $y^2=x^3+t^5+1$ has $6$ cuspidal fibers (see \cite[Example 5.8]{Schutt_Schioda_MordellWeil}). Thus $2a+m-4=8$, yielding $\MW_0(\pi)\cong \Z^8$. Thus, $\pi$
    has the same narrow Mordell-Weil group as the generic rational elliptic surface, which has $12$ nodal fibers.
\end{ex}

\begin{ex} [Elliptic fibrations with $\MW_0(\pi)\cong \Z^4$] 
Consider an elliptic fibration $\pi\colon M\to \p^1$.
By \Cref{cor: rk MW_0 thm}, the following configurations of $m$ multiplicative fibers and $a$ additive fibers yield $\MW_0(\pi)\cong \Z^4$:
\begin{center}
\begin{tabular}{|c|c|c|c|c|c|}
\hline
 \textbf{m} &  8 & 6 & 4 & 2 & 0\\ 
 \hline
 \textbf{a} & 0 & 1 & 2 & 3 & 4 \\  
\hline
\end{tabular}
\end{center}
By referring to Persson's list of configurations of singular fibers on a rational elliptic surface \cite{Persson_Configurations}, we see that these all occur. See \Cref{fig:2I_3} and \Cref{fig:4III} for two examples.

We also give an example of a K3 surface with $\MW_0(\pi)\cong \Z^4$. Let $E_1,E_2$ be elliptic curves and let $\iota\colon E_1\times E_2\to E_1\times E_2$ be the product of their hyperelliptic involutions. Then $M\coloneqq \mathrm{Km}(E_1\times E_2)$ is the minimal resolution of $(E_1\times E_2)/\iota$. Geometrically, this can be thought of as replacing a neighborhood of each of the 16 singular points of $(E_1\times E_2)/\iota$ with spheres of self-intersection $-2$.
The projection $E_1\times E_2\to E_1$ induces a map $M\to \p^1\cong E_1/\{\pm 1\}$, giving $M$ the structure of an elliptic fibration with 4 singular fibers of type $I_0^*$, see \Cref{fig:Kummer}.
\end{ex}

\begin{ex}[\textit{$\rk(\MW_0)$ can distinguish between different elliptic fibrations on the same complex surface}]
For two non-isogenous elliptic curves $E_1,E_2$ with involution $\iota\colon E_1\times E_2\to E_1\times E_2$, the Kummer surface $X$ is the minimal resolution of $(E_1\times E_2)/\iota$. Oguiso showed that (up to equivalence) such a Kummer surface admits $11$ different elliptic fibrations with different configurations of singular fibers (see \cite[Table A]{Oguiso_DiffFibOnK3}). The Kummer surface shown in  \Cref{fig:Kummer} has four singular fibers of additive type and thus $\MW_0(\pi)\cong \Z^4$.
According to Oguiso's table, there also exist fibrations with $10$ multiplicative fibers (e.g. $2$ fibers of type $I_8$ and $8$ fibers of type $I_1$). These have $\MW_0(\pi)\cong \Z^6$.
\end{ex}

\printbibliography

\noindent
\textsc{Department of Mathematics, University of Chicago} \\
\textit{Email:} \href{mailto:mmorariu@uchicago.edu}{mmorariu@uchicago.edu}
\end{document}